\documentclass[11pt]{amsart}
\usepackage{fullpage}
\usepackage{amssymb,amsmath,amsfonts,amsthm,mathrsfs,moreverb}
\usepackage[all]{xy} 
\SelectTips{cm}{10}
\usepackage{color}

\numberwithin{equation}{section}
\newcommand{\KK}{\mathbb K}

\def\I{{\operatorname{I}}}

\def\ord{\operatorname{ord}}

\def\bbar#1{\setbox0=\hbox{$#1$}\dimen0=.2\ht0 \kern\dimen0 
\overline{\kern-\dimen0 #1}}

\newtheorem{thm}{Theorem}[section]
\newtheorem{lemma}[thm]{Lemma}
\newtheorem{cor}[thm]{Corollary}
\newtheorem{prop}[thm]{Proposition}

\theoremstyle{definition}
\newtheorem{definition}[thm]{Definition}

\newtheorem{notation}[thm]{Notation}

\newtheorem{remark}[thm]{Remark}

\newenvironment{romanenum}{\hfill \begin{enumerate} }{\end{enumerate}}

\definecolor{webcolor}{rgb}{0.8,0,0.2}
\definecolor{webbrown}{rgb}{.6,0,0}
\usepackage[
        colorlinks,
        linkcolor=webbrown,  filecolor=webcolor,  citecolor=webbrown, 
        backref,
        pdfauthor={Yansong Xu}, 
       pdftitle={On split of minor roots },
]{hyperref}
\usepackage[alphabetic,backrefs,lite]{amsrefs} % for bibliography

\begin{document}
\title[two dimentional jacobian conjecture]{Intersection numbers and split of minor roots}

%14R15 Jacobian problem

\author{Yansong Xu}
\address{}
\email{yansong\_xu@yahoo.com}

\begin{abstract}
There are two main parts in this manuscript. First, 
for a Jacobian pair $(f, g)$, with the concept of final major roots and final minor roots, we obtain equations  and inequalities for the intersection numbers $\I (f_\xi, g)$ and  ${\I}(f_\xi,f_y)$ respectively, here $\xi$ is a generic element of the base field and $f_\xi=f-\xi$ and discuss usage of them in some special cases.
 Second, we discuss all possibilities of the splits of principal minor roots for the case of degree (99, 66) with help of Abhyankar-Moh planar semigroup, find an unknown possible split and suggest case (99, 66) is open.
\end{abstract}

\maketitle

%-----------------------------------------xxxxx

\section{Introduction}
Through out this paper, $\KK$ denotes an algebraically closed field of characteristic 0.

Let $(f(x,y), g(x,y))$ be a Jacobian pair. The Jacobian condition $f_x g_y-f_y g_x=J$ is an expression of two terms. We take a root of $f(y)f_y(y)$ in the Puiseux field $\KK\ll x^{-1} \gg$ and plug in the expression. The expression becomes one term, therefore the orders of the factors of the expression are calculable in conditions.

Moh in \cite{moh} first introduces the concepts of $\pi$-root, major root and minor root. He gets plentiful results for major roots and a few for minor roots. For minor roots, Moh proves that any minor roots are not split before order 1. We introduce the concept of final roots. We find some relationship between intersection numbers and the split of final minor roots. On the other hand,  ${\I}(f_\xi, g)$ can be expressed by only final major roots  while ${\I}(f_\xi,f_y)$ can also be expressed by total split $\pi$-roots. These formulas and inequalities of final major roots and final minor roots provide some constrains for possible counter examples for Jacobian Conjecture.

With the help of quasi-approximate roots and Abhyankar-Moh planar semi-group theory, we show that for any case, the principal minor roots are not splitting at order 1 for all effective quasi-approximate roots. This provide us a tool to further study the split possibilities for the principal minor roots.

In the last section, we discuss the split possibilities for case (99, 66) that Moh does not discuss in details and find an unknown possible split and suggest case (99, 66) is open.

\section{Preliminaries}

For polynomials $f(x,y), g(x,y)\in \KK[x,y]$, monic in $y$, we define intersection number of $f$ with $g$ as $\I(f,g)=\deg_x{\rm Res}_y(f,g)$.
 Let $x=t^{-1}$. We simply write $f(t^{-1}, y)$ as $f(y)$. 
$f(y)$ can be factored out completely over Puiseux field $\KK\ll t\gg$.

Let $\pi$ be a symbol. Following \cite{moh}, we define $\sigma=\sum_{j<\delta_\sigma}a_jt^j+\pi t^{\delta_\sigma}$ as a $\pi$-root of $f$ if $f(\sigma)=f_\sigma(\pi)t^{\lambda_{\sigma}}+\cdots$
and $\deg f_\sigma(\pi)>0$. The multiplicity of $\sigma$ is defined as $\deg_\pi f_\sigma(\pi)$ \cite{moh}. We call $\sigma$ is split
 if $f_\sigma(\pi)$ has more than one roots. We call $\delta_\sigma$ the split order of $\sigma$, or we say $ f$ splits  at order $\delta_\sigma$.  We call $\sigma$  is final  if $f_\sigma(\pi)$ has no multiple roots and $\deg_\pi f_\sigma(\pi)>1$. We also call a final $\pi$-root as a final root for simplicity.

Let $\alpha=\sum a_j t^j \in \KK \ll t \gg$ and let $\delta$ be a real number, we denote $\alpha |_{<\delta}=\sum_{j<\delta} a_j t^j$. If $\beta=\alpha |_{<\delta}$, then we say $\alpha$ is an extension of $\beta$.
Let $\sigma=\sum_{j<\delta_\sigma}a_jt^j+\pi t^{\delta_\sigma}$ be a $\pi$-root, we denote  $\sigma |_{\pi=b}=\sum_{j<\delta_\sigma}a_jt^j+b t^{\delta_\sigma}$.

 Let $f(x, y)$ and $g(x, y)$ be polynomials in $\KK[x, y]$, monic in $y$. We call $(f, g)$ a Jacobian pair if $f_x g_y-f_y g_x \in \KK^*$. Let $\xi$ be a generic element in $\KK$ and let $f_\xi=f-\xi$. The following proportional lemma is well-known.
\begin{lemma}\label{L:proportional}
Let $(f, g)$ be a Jacobian pair.  Let  $\deg_y f=m$ and $\deg_y g=n$.
Let $\sigma$ be a $\pi$-root and let 

$$ f_\xi (\sigma)=f_{\sigma}(\pi) t^{\lambda^f}+\cdots, \quad    g (\sigma)=g_{\sigma}(\pi) t^{\lambda^g}+\cdots.$$
Then we have
\begin{romanenum}
\item \label{L:prop i} If $\lambda^f<0$ and  $\deg_\pi f_\sigma(\pi)>1$, then
$\lambda^f:\lambda^g=m:n=\deg_\pi f_{\sigma}(\pi):\deg_\pi g_{\sigma}(\pi)$.
\item \label{L:prop ii}
If $ f_{\sigma}(\pi)$ has multiple roots, then
$f_{\sigma}(\pi) ^n/  g_{\sigma}(\pi)^m \in \KK^*.$
\end{romanenum}

\end{lemma}

\section{Root Splitting and  $ \I(f_\xi , f_y)$}

In this section, we do not need the Jacobian condition. The results can be applied to other topics.
\begin{lemma}\label{L:derivative}
Let $p(\pi)$ be a polynomial in $\KK[\pi]$ and $p(\pi)$ have $e$ different roots. Then the derivative  $p^{\prime}(\pi)$ has exactly $e-1$ roots that are not roots of $p(\pi)$.
\end{lemma}
\begin{proof}
Let $p(\pi)=a(\pi-a_1)^{n_1}\dots (\pi-a_e)^{n_e}$. Then  $p^{\prime}(\pi)=p_1(\pi)(\pi-a_1)^{n_1-1}\dots (\pi-a_e)^{n_e-1}$. It is easy to see that $\deg p_1(\pi)=e-1$ and that $p_1(\pi)$ and $p(\pi)$ have no common roots.
\end{proof}

\begin{notation} Let $p(\pi)$ be a polynomial in $\KK[\pi]$. Then $e(p(\pi))$ denotes the number of distinct roots of $p(\pi)$.
\end{notation}

The following proposition shows the structure of the roots of  $f_y$ by split $\pi$-roots of $f_\xi $.
\begin{prop}\label{P:fy}
 Let  polynomial $f(x,y)$ be monic in $y$ and $f_\xi$ have Puiseux expansion 
$$  f_\xi( y)=f_\xi(t^{-1}, y)=\prod_{i=1} ^{m}(y-\alpha _i) $$
 over $\KK\ll t\gg$ and let $\sigma=\sum_{j<\delta}a_j t^j + \pi t^\delta$
be a $\pi$-root of $f_\xi(y)$. Suppose 
$$f_\xi(\sigma)=f_{\sigma}(\pi)t^{\lambda_{\sigma}} + \operatorname{higher} \operatorname{terms} \operatorname {in } t$$
and $\deg_{\pi}f_{\sigma}(\pi)\geq1$. Then we have

\begin{romanenum}
\item \label{L:fy i}
$f_y( \sigma )=f_{\sigma}' (\pi) t^{{\lambda_{\sigma}}-\delta} + {\rm higher\; terms\; in\; } t$.
\item \label{L:fy ii}
If $\sigma$ is a split $\pi$-root of $f_\xi$,  then $\sigma$ is also a $\pi$-root of $f_y(y)$.
\item \label{L:fy iii}
If $\sigma$ is a split $\pi$-root of $f_\xi$, then there are exactly $e(f_{\sigma}(\pi))-1$ roots $\beta_1, \dots, \beta_{e(f_{\sigma}(\pi))-1}$ of $f_y$  such that
$$\ord f_\xi(\beta_i)=\lambda_{\sigma} \quad  for \; i=1, \dots, e(f_{\sigma}(\pi))-1.$$
\item \label{L:fy vi}
For any root $\beta$ of $f_y$, there is a split $\pi$-root $\sigma_\beta$ of $f_\xi$, such that
$$\ord f_\xi(\beta)=\lambda_{\sigma_\beta}.$$
$\sigma_\beta$ and $\beta$ are said related each other.
\item \label{L:fy v}
Let $B$ be the root set of $f_y$. For any split $\pi$-root $\sigma$ of $f_\xi$, let
$$B_{\sigma}=\{\beta | \beta \in B  \; and \;\ord f_\xi(\beta)=\lambda_{\sigma}\},$$
then
$$B=\biguplus_{\sigma} B_{\sigma}$$
\noindent here $\sigma$ runs through all split $\pi$-root of $f_\xi$.
\end{romanenum}
\end{prop}

\begin{proof}
(\ref{L:fy i}) Let $\deg_{\pi}{f_{\sigma}(\pi)}=s$. Reorder index if necessary, we can suppose that 
$$\sigma - \alpha _i=(\pi-c_i)t^\delta + \operatorname{higher}  \operatorname{terms}  \operatorname{in}t\;  \quad \operatorname{for}\;  i=1,\dots, s$$
and 
$$\sigma - \alpha _i=b_it^{l_i}+{\rm higher\;  terms\; in }\;t \quad {\rm for}\;  i=s+1,\dots, m$$
\noindent here $l_i<\delta$ for $ i=s+1,\dots, m$.
Then we have

\begin{align*}
f_\xi(\sigma)&=\prod _{i=1}^{m}(\sigma-\alpha_i)\\
&=\prod _{i=1}^{s}((\pi-c_i)t^{\delta})\prod _{i=s+1}^{m}b_it^{l_i}
+ \cdots \\
&=f_{\sigma}(\pi)t^{\lambda_{\sigma}} 
+ \cdots.
\end{align*}

As
$$f_y(y)=\sum_{i=1}^m\prod_{j\neq i}(y-\alpha_j)
=\sum_{i=1}^m{\prod_{j=1}^m(y-\alpha_j) \over y-\alpha_i}
$$

we have
\begin{align*}
f_y(\sigma)
&=\sum_{i=1}^m{\prod_{j=1}^m(\sigma-\alpha_j) \over \sigma-\alpha_i}\\
&=\sum_{i=1}^s{\prod_{j=1}^m(\sigma-\alpha_j) \over \sigma-\alpha_i}
+ \cdots\\
&=\left( \prod_{i=s+1}^mb_it^{l_i}\right) \sum_{i=1}^s{\prod_{j=1}^s(\pi-c_j) \over \pi-c_i}t^{(s-1)\delta}
+ \cdots\\
&=f_{\sigma}' (\pi)t^{{\lambda_{\sigma}} -\delta}
+ \cdots.
\end{align*}
(\ref{L:fy ii}) As $\deg_\pi f_\sigma(\pi)>1$, therefore $\deg_\pi f_\sigma'(\pi)\geq 1$. By definition, $\sigma$ is a $\pi$-root of $f_y(y)$.

\noindent (\ref{L:fy iii}) Let $k=e(f_{\sigma}(\pi))$ and let $f_{\sigma}(\pi)=c(\pi-c_1)^{n_1}\cdots (\pi-c_k)^{n_k}$. By Lemma \ref{L:derivative},

$$ f_\sigma'(\pi)=f_1(\pi)(\pi-c_1)^{n_1-1}\cdots (\pi-c_k)^{n_k-1} $$
with $\deg f_1(\pi)=k-1$ and also $f_1(\pi)$ and 
$f_{\sigma}(\pi)$ have no common roots. Let $f_1(\pi)=b(\pi-b_1)\cdots (\pi-b_{k-1}).$
By (\ref{L:fy ii}), $\sigma$ is $\pi$-root of $f_y$. By \cite{moh} Proposition 1.2, there are exactly $k-1$ roots 
$\beta_1,\dots,\beta_{k-1}$ of $f_y$ such that
 $$\beta_i=\sum_{j<\delta}a_j t^j+b_i t^\delta+\cdots, \quad \ord f_\xi(\beta_i)=\lambda_\sigma \quad {\rm for\; }i=1,\dots,k-1.$$
\noindent (\ref{L:fy vi}) As $\xi$ is a generic element of the base field $\KK$, $f_\xi(\beta)\neq 0$. Let 
$$\delta_\beta=\max \{\ord (\alpha - \beta )| f_\xi(\alpha)=0\}$$
Write $\beta=\sum b_j t^j$ and let $\sigma_\beta=\sum_{j<\delta_\beta} b_j t^j+\pi t^{\delta_\beta}$, then $\sigma_\beta$ is a $\pi$-root of $f_\xi$ and $\ord f_\xi(\sigma_\beta)=\ord f_\xi(\beta)$.

\noindent (\ref{L:fy v}) This is directly result of (\ref{L:fy iii}) and (\ref{L:fy vi}).
\end{proof}

The intersection number $ \I(f_\xi , f_y)$ can be expressed by all split $\pi$-roots.
\begin{thm} \label{T:main}
Let $f(x,y)$ be a polynomial in $\KK[x, y]$, monic in $y$. Then 
$$ \I(f_\xi , f_y) =-\sum_{\sigma} (e(f_\sigma(\pi))-1){\lambda_\sigma}$$
\noindent here  $\sigma$ goes through all splitting $\pi$-roots  of $f_\xi$ and $f_\xi(\sigma)=f_{\sigma}(\pi)t^{\lambda_\sigma}+\cdots$.
\end{thm}
%xxxxxxxx
\begin{proof}
Let $B=\{\beta|f_y(\beta)=0\}$. By Proposition \ref{P:fy}(\ref{L:fy v}), $B=\biguplus_{\sigma} B_{\sigma}$. Therefore
$${\I}(f_\xi , f_y)=- \sum_{\beta \in B}\ord f_\xi(\beta)=-\sum_{\sigma}\sum_{\beta \in B_\sigma}\ord f_\xi(\beta) = -\sum_\sigma (e(f_\sigma(\pi))-1){\lambda_\sigma}.$$
\end{proof}

\section{ $ \I(f_\xi , f_y)$ and $ {\I}(f_\xi , g)$ and  Final Minor Roots}
The following lemma is simple but fundamental to our calculations.

\begin{lemma}\label{L:simple}
Let $f, g \in \KK[x, y]$. Let $f_xg_y-f_yg_x=J(x, y)$. And let $\alpha$ be an element of the Puiseux field $\KK\ll t\gg$. Then we have
\begin{equation}\label{L:simple1}
g_y(\alpha)\dfrac{d}{dt}f(\alpha) -f_y(\alpha)\dfrac{d}{dt}g(\alpha)=-J(t^{-1}, \alpha)t^{-2}.
\end{equation}

\end{lemma}
\begin{proof}
By chain rule,
\begin{equation}\label{L:simple2}
\dfrac{d}{dt}f(\alpha)=f_x(\alpha)\dfrac{dx}{dt}+f_y(\alpha)\dfrac{d\alpha}{dt}, \quad 
\dfrac{d}{dt}g(\alpha)=g_x(\alpha)\dfrac{dx}{dt}+g_y(\alpha)\dfrac{d\alpha}{dt}.
\end{equation}
Substituting (\ref{L:simple2}) into  LHS of (\ref{L:simple1}), simplifying and noting $dx/dt=-t^{-2}$ we get  RHS of (\ref{L:simple1}).
\end{proof}

From now on, we suppose that $(f, g)$ is a Jacobian pair.

We have the following lemma for final root.
\begin{lemma}\label{L:finalRoot}
Let   $\alpha=\sum a_j t^j$ be a root of $ f_\xi (y)$. Let
$$\delta=\max \{\ord(\alpha-\beta)\:|\: g(\beta)=0\}.$$
And let $\sigma=\sum_{j<\delta}a_j t^j+\pi t^\delta$. Then $\sigma$ is a final $\pi$-root. The final $\pi$-root $\sigma$ and the root $\alpha$ are called related each other.
\end{lemma}

\begin{proof}
Let
$$ f_\xi (\sigma)=f_{\sigma}(\pi) t^{\lambda^f}+\cdots, \quad    g (\sigma)=g_{\sigma}(\pi) t^{\lambda^g}+\cdots.$$
If $\sigma$ is not final, then by the proportional lemma (Lemma \ref{L:proportional}), $f_{\sigma}(\pi) ^n/ g_{\sigma}(\pi)^m\in \KK^*$. This is a contradiction to the definition of $\delta$.

\end{proof}

\begin{definition}
Let  $\alpha$ be a root of $ f_{\xi} (y)$. If $ \ord g(\alpha)<0$ then  $\alpha$ is called a major root of $ f_{\xi} (y)$;
if  $ \ord g(\alpha)=0$ then  $\alpha$ is called a minor root of $ f_{\xi} (y)$.
When $f(y)$ has two points at infinity with the leading form $(y-a_1x)^u (y-a_2x)^v$ and $u<v$, all roots $\alpha=a_1 t^{-1}+\cdots$ of $f_\xi(y)$ are minor and the total collection of them are called the principal minor roots of $f_\xi(y)$.
\end{definition}

\begin{lemma}\label{L:rootorder}
Let  $\alpha$ be a root of $ f_{\xi} (y)$. Let $\sigma=\sum_{j<\delta}a_j t^j+\pi t^\delta$ be the related final $\pi$-root.
Then we have
\begin{romanenum}
\item \label{L:mm i}
If $\alpha$ is a major root, then $\delta<1$
\item \label{L:mm ii}
If $\alpha$ is a minor root, then $\delta>1$.
\end{romanenum}
\end{lemma}

\begin{proof}
Let 
 $$ f_{\xi} (\sigma)=f_{\sigma}(\pi) t^{\lambda^f}+f_1(\pi) t^{\lambda^f+\varepsilon}  +\cdots ,\quad    g (\sigma)=g_{\sigma}(\pi) t^{\lambda^g}+g_1(\pi) t^{\lambda^g+\varepsilon}  +\cdots $$
here $\varepsilon >0$.

(\ref{L:mm i}) Since $\alpha$ is major, then $\lambda^f<0$ and  $\lambda^g<0$.
By  \cite{moh} Proposition 4.1, 
$$ (\lambda^f f_{\sigma}(\pi) g'_{\sigma}(\pi) - \lambda^g f'_{\sigma}(\pi) g_{\sigma}(\pi)) t^{\lambda^f+\lambda^g-1}+\cdots=\dfrac{\partial (f(\sigma),g(\sigma))}{\partial (t, \pi)}=-J t^{\delta-2}.$$
Since $\sigma$ is final, therefore $ \lambda^f f_{\sigma}(\pi) g'_{\sigma}(\pi) - \lambda^g f'_{\sigma}(\pi) g_{\sigma}(\pi) \neq 0$. Hence $\lambda^f+\lambda^g-1=\delta-2$. Therefore  $\delta<1$.

(\ref{L:mm ii}) Now $\lambda^f=0= \lambda^g$. We have
$$ (\varepsilon f_1(\pi) g'_{\sigma}(\pi) - \varepsilon f'_{\sigma}(\pi) g_1(\pi)) t^{\varepsilon-1}+\cdots=\dfrac{\partial (f(\sigma),g(\sigma))}{\partial (t, \pi)}=-J t^{\delta-2}.$$
Hence $\varepsilon-1\leq \delta-2$. Therefore $\delta>1$.
\end{proof}

Now we can show the relationship between our minor roots and  \cite{moh} minor discs of a split $\pi$-root.

\begin{cor}\label{C:c31}
Let $\tau =\sum_{j<\delta_r} a_j t^j+c_r t^{\delta_r}$ be a minor disc center defined as in \cite{moh} Proposition 6.1. Let 
$$E^f_\tau=\{\alpha|f_\xi (\alpha)=0  \rm{ \;\, and\; } \ord (\tau-\alpha)>\delta_r)\}.$$
Then all elements in $E^f_\tau$ are minor roots.
\end{cor} 

\begin{proof}
Let $\alpha \in E^f_\tau$, and let
$$\delta^{*}=\min \{ \ord(\alpha-\beta)|\beta {\rm \;is \; root \; of \;} f_\xi(y)g(y) {\rm \;\; and\; } \ord (\tau-\beta)>\delta_r\}.$$
Then by \cite{moh} Proposition 6.1, $\delta^*\geq 1$. According to Lemma \ref{L:rootorder}, $\alpha$ is a minor root.
\end{proof}

\begin{definition} 
Let
 $\sigma =\sum_{j<\delta_\sigma} a_j t^j+\pi  t^{\delta_{\sigma}}$ be a final $\pi$-root of $f_\xi$.
Let polynomial $h(y) \in \KK [x,y]$. We define set
$$D^h_{\sigma}=\{ \alpha \: |\: h(\alpha)=0  \rm {\; and \;} \ord (\sigma-\alpha)=\delta_{\sigma} \}.$$
If $h(\alpha)=0$ and $\alpha \in  D^h_{\sigma}$, we call $\alpha$ and $\sigma$ are related each other. We denote $|D^h_{\sigma}|$ be the cardinal of set $D^h_{\sigma}$. 
We define sets
$$P_m=\{\sigma \: |\: \sigma  {\rm\; is \;final \; minor\; } \pi\operatorname {-root \; of \;}  f_\xi\} ,$$
$$P_M=\{\sigma \: |\: \sigma  {\rm \; is \;final \; major \;} \pi   \operatorname{-root \;of \;} f_\xi\}.$$

\end{definition}

Intersection numbers $ \I (f_{\xi} , g)$ and 
$ {\I}(f_\xi , f_y)$  have relationships with final minor $\pi$-roots.

\begin{thm} \label{T:mainminor}
Let $(f, g)$ be a Jacobian pair. Then we have
\begin{romanenum}
\item \label{T:m i}

$$ {\I}(f_\xi , f_y) \le \deg_y f -1+\sum_{\sigma \in P_m}(|D^{f_\xi}_{\sigma}|-1)(\delta_{\sigma}-1)$$
\item \label{T:m ii}
$$ {\I} (f_\xi , g)\ge 1+\sum_{\sigma \in P_m}(\delta_{\sigma}-1).$$

If there are no final minor roots, then
\item \label{T:m iii}
$${\I}(f, f_y)=\deg_y f-1$$
$${\I}(f, g)=1.$$
\end{romanenum}
\end{thm}

\begin{proof}
Let $f_\xi(\alpha)=0$. By Lemma \ref{L:simple}, 
$$-f_y(\alpha)\dfrac{d}{dt}g(\alpha)=g_y(\alpha)\dfrac{d}{dt}f(\alpha) -f_y(\alpha)\dfrac{d}{dt}g(\alpha)=-Jt^{-2}.$$
If $\alpha$ is a major root, then $\ord g(\alpha)<0$. Therefore 
$$\ord \dfrac{d}{dt}g(\alpha)=\ord g(\alpha)-1.$$
Thus
$$\ord f_y(\alpha)g(\alpha)=-1.$$
If $\alpha$ is a minor root, then $\ord g(\alpha)=0$. Define
$$\delta = \max \{\ord (\beta -\alpha)| g(\beta)=0\}.$$
By Proportion \ref{P:fy}, $\ord f_y(\alpha)=-\delta$. Therefore
$$\ord f_y(\alpha) g(\alpha)=-\delta=-1-(\delta-1).$$
Thus

\begin{align}\label{eq:m1}
{\I} (f_\xi , f_yg)&=-\ord \prod_{\{\alpha|f_\xi(\alpha)=0\}} f_y(\alpha)g(\alpha)\\
&=-\sum_{\sigma \in P_m}\sum_{\alpha \in D^{f_\xi}_\sigma} \ord f_y(\alpha)g(\alpha)
	-\sum_{\sigma \in P_M}\sum_{\alpha \in D^{f_\xi}_\sigma} \ord f_y(\alpha)g(\alpha)\notag \\ 
&=\sum_{\sigma \in P_m}|D^{f_\xi}_\sigma|(1+(\delta_\sigma-1))+\sum_{\sigma \in P_M}|D^{f_\xi}_\sigma| \notag \\
&=\deg_y f +\sum_{\sigma \in P_m}|D^{f_\xi}_\sigma|(\delta_\sigma-1). \notag
\end{align}
Now let $f_y(\beta)=0$. Note that $\ord f_\xi(\beta)=0$ if and if $\beta \in D^{f_y}_\sigma$ for some $\sigma \in P_m$ and $\ord g_y(\beta)\ge -\delta _\sigma$. By Proposition \ref{P:fy}, $|D^{f_y}_\sigma|= |D^{f_\xi}_\sigma|-1$. With the similar reasons as above, we have
\begin{align}\label{eq:m2}
{\I} (f_{y} , f_\xi g_y)&=-\ord \prod_{\{\beta|f_y(\beta)=0\}} f_\xi(\beta)g_y(\beta)\\
&\le \deg_y f_y +\sum_{\sigma \in P_m}|D^{f_y}_\sigma|(\delta_\sigma-1) \notag \\
&=\deg_y f-1 +\sum_{\sigma \in P_m}(|D^{f_\xi}_\sigma|-1)(\delta_\sigma-1). \notag
\end{align}
As $\I (f_y, f_\xi g_y)=\I (f_y, f_\xi)+\I (f_y, g_y)$, and from the Jacobian condition, ${\I} (f_y , g_y)=0$, (\ref{T:m i}) is proved.
Note $\I (f_\xi , f_yg)=\I (f_\xi , f_y)+\I (f_\xi , g)$ and subtract (\ref{eq:m2})  from (\ref{eq:m1}) ,  (\ref{T:m ii}) is proved.
If there are no final minor roots, then all above inequalities become equal and (\ref{T:m iii}) follows.
\end{proof}

\begin{cor}\label{C:equivalent}
Let $(f,g)$ be a Jacobain pair. The following are equivalent.
\begin{romanenum}
\item \label{C:m i}
$\KK (f,g)=\KK(x,y)$.
\item \label{C:m ii}
$\KK [f,g]=\KK[x,y]$.

\item \label{C:m iii}

${\I} (f , g)=1$.
\item \label{C:m iv}
$f$ has no minor roots.

\item \label{C:m v}
$ {\I}(f , f_y)=\deg_y f -1$.
\end{romanenum}

\end{cor} 
\begin{proof}
Let $\xi$ and $\eta$ be symbols, let $f_\xi=f-\xi$, $g_\eta=g-\eta$. And let 
$F(\xi,\eta,x)={\rm Res}_y(f_\xi, g_\eta)$. Then $F(\xi,\eta,x)$ is irreducible (\cite{moh} page 150) and $F(f(x,y), g(x,y),x)=0$. So $F(f,g,X)$ is the define equation of $x$ over $\KK (f,g)$. Thus ${\I} (f_{\xi} , g_\mu)$ is the degree of field extension 
$[\KK (x,y):\KK (f,g)]$, which is equal to 1 if and only if there is no minor root from above theorem. We also see from the proof of the last theorem, if there is no minor root, then the coefficient of the leading term of $F(\xi , \mu, x)$ in $x$ is in $\KK^*$. The corollary is thus proved.
\end{proof}

\begin{cor}\label{C:c33}{\rm (\cite{xu})}
Let $f(x,y)$ be a polynomial, monic in y. If f is an embedded line,  then  ${\I}(f , f_y)=\deg_y f -1$. In other words, the curve f has $\deg_y f -1$ vertical tangent lines.
\end{cor}

\begin{proof}
By Abhyankar-Moh's Epimorphism Theorem (\cite{am}), there exits polynomial $g(x,y)$ such that Jacobian $J(f,g) \in \KK^*$
and $\KK [x,y]=\KK [f,g]$. By above corollary, $ {\I}(f , f_y)=\deg_y f -1$.

\end{proof}

\section{$ {\I}(f_\xi , g)$ Expressed by Final Major Roots}
Now we express $ {\I}(f_\xi , g)$ with final major $\pi$-roots. We define

\begin{align*}
 {\I}_M(f, g)&=-\sum_{\sigma \in P_M}|D^{f}_{\sigma}|\lambda^g_\sigma\\
&=\dfrac{\deg_yg}{\deg_yf+\deg_yg}\sum_{\sigma \in P_M}|D^{f}_{\sigma}|(1-\delta_\sigma).
\end{align*}

\begin{thm} \label{T:main4}
Let $(f, g)$ be a Jacobian pair. Then we have
$$ {\I}(f_\xi , g)= {\I}_M(f , g)$$
\end{thm}

\begin{proof}
Let $\deg_y f=m$ and  $\deg_y g=n$.
Let $f_\xi(\alpha)=0$. 
If $\alpha$ is a major root, then $\ord g(\alpha)<0$. Let $\sigma$ be its final $\pi$-root. And let
$$ f(\sigma)=f_\sigma(\pi) t^{\lambda^f_\sigma}+\cdots, \quad g(\sigma)=g_\sigma(\pi) t^{\lambda^g_\sigma}+\cdots.$$
By \cite{moh} Proposition 4.6, 
$$-\lambda^g_\sigma=n\dfrac{1-\delta_\sigma}{n-M_1}=\dfrac{n}{m+n}(1-\delta_\sigma).$$
If $\alpha$ is a minor root, then $\ord g(\alpha)=0$. Therefore

\begin{align*}
 {\I}(f_\xi , g)&=-\ord \prod_{\{\alpha | f_\xi(\alpha)=0\} }g(\alpha)\\
&=-\sum_{\sigma \in P_M}\sum_{\alpha \in D^{f}_{\sigma} }\ord g(\alpha)     -\sum_{\sigma \in P_m}\sum_{\alpha \in D^{f_\xi}_{\sigma} }\ord  g(\alpha)\\
&=-\sum_{\sigma \in P_M}|D^{f}_{\sigma}|\lambda^g_\sigma\\
&=\dfrac{n}{m+n}\sum_{\sigma \in P_M}|D^{f}_{\sigma}|(1-\delta_\sigma).
\end{align*}

\end{proof}

Zhang \cite{zhang} gives estimate for the degree of the field extension $[\, \KK(x,y):\KK(f,g)] \le \min(\deg f , \deg g)$. Using the theorem above, we can give a better estimate.

\begin{cor}\label{C:c41}
Let $(f,g)$ be a Jacobian pair. Let $\deg_y f >1$ and  $\deg_y g >1$. Then the field extension degree $[\, \KK(x,y):\KK(f,g)]<\dfrac{\deg_yf \deg_yg}{\deg_yf+\deg_yg}$.
\end{cor} 
\begin{proof}
We only need note that in the proof of the last theorem, $0<\delta_\sigma<1$  (\cite{moh} Proposition 5.3) and 
$$\sum_{\sigma \in P_M}|D^{f}_{\sigma}|(1-\delta_\sigma) < \sum_{\sigma \in P_M}|D^{f}_{\sigma}|\le \deg_y f.$$
\end{proof}

We define
$${\I}_m(f, g)=1+\sum_{\sigma \in P_m}(\delta_\sigma-1).$$
By combining Theorem \ref{T:mainminor} and Theorem \ref{T:main4}, we have

\begin{cor}\label{C:c42}
Let $(f,g)$ be a Jacobian pair. Then 
$${\I}_M(f, g)\ge {\I}_m(f, g).$$

\end{cor} 
%\begin{proof}
%\end{proof}

\section{Applications}
We apply Corollary \ref{C:c42} to special cases. We show that these formulas work in some cases but are not effect in other cases. 

\subsection{Case (75, 50)} \label{SS:case75} We take case (75, 50) as an example.
According to \cite{moh}, page 202, there are two sub cases  [For (75, 50) second sub case, $\delta_1=1/3$ should be $\delta_1=2/3$]. We only show the second one here, as it provides more splitting possibilities, therefore more complicated. i.e. We have 
$$n=75, \quad m=50, \quad M_2=55, \quad M_3=73, \quad V_3=4, \quad V_2=2, \quad \delta_2=1/5, \quad \delta_1=2/3.$$
Let $\sigma_2$ be the major $\pi$-root with the split order $\delta_2$. There are two possibilities of splittings for $\sigma_2$.
\begin{romanenum}
\item \label{L:case75 i} 
$$f(\sigma_2)=a((\pi^5-c_1)(\pi^5-c_2))^4 t^{-2}+\cdots, \quad g(\sigma_2)=b((\pi^5-c_1)(\pi^5-c_2))^6 t^{-3}+\cdots.$$
All these 10 roots of $f_{\sigma_2}(\pi)$ and $g_{\sigma_2}(\pi)$ are major. Each of them extends  to final major $\pi$-roots $\sigma_{1i}$ with split order $\delta_1$:
$$f(\sigma_{1i})=a_{0i}\pi (\pi^3-a_{1i}) t^{-2/15}+\cdots, \quad g(\sigma_{1i})=b_{0i} (\pi^3-b_{1i})(\pi^3-b_{2i} ) t^{-1/5}+\cdots $$
for $  i=1,\dots,10.$
We have only one final minor $\pi$-root with split order $V_3/U_3=4$. Therefore,
from final minor roots, we have $ {\I_m}(f, g) = 1+(4-1)=4$, while final major roots give

$ {\I_M}(f, g)=10 \times 4 \times 1/5=8$.

\noindent They do not work in this case.

\item \label{L:case75 ii}
$$f(\sigma_2)=a((\pi^5-c_1)^2(\pi^5-c_2) (\pi^5-c_3))^2 t^{-2}+\cdots,$$
$$g(\sigma_2)=b((\pi^5-c_1)^2(\pi^5-c_2) (\pi^5-c_3))^3 t^{-3}+\cdots.$$
In this case, 5 roots from $(\pi^5-c_1)$ are major and extend to final major $\pi$-roots $\sigma_{1i}$ with split order $\delta_1$:
$$f(\sigma_{1i})=a_{0i}\pi (\pi^3-a_{1i}) t^{-2/15}+\cdots, \quad g(\sigma_{1i})=b_{0i} (\pi^3-b_{1i})(\pi^3-b_{2i} ) t^{-1/5}+\cdots $$
for $  i=1,\dots,5;$
while 10 roots from $(\pi^5-c_2) (\pi^5-c_3)$ are minor and extend to final minor $\pi$-roots with split order  $6/5$. So we have

$ {\I_m}(f , g)= 1+(4-1)+10\times (6/5-1)=6$ from final minor roots while

$ {\I}_M(f , g)=5 \times 4 \times 1/5=4$.

\noindent This is a contradiction and we can rule out this case.
\end{romanenum}

\subsection{Case (84, 56)} \label{SS:case84}  From \cite{moh} page 202, there are two cases.
\begin{romanenum}
\item \label{L:case84 i} 
$$n=84, \quad m=56, \quad M_2=64, \quad M_3=82, \quad V_3=3, \quad V_2=2, \quad \delta_2=2/7, \quad \delta_1=16/21.$$
Let $\sigma_2$ be the major $\pi$-root with the split order $\delta_2$. 
$$f(\sigma_2)=a((\pi^7-c_1)^2(\pi^7-c_2))^2 t^{-2}+\cdots, \quad g(\sigma_2)=b((\pi^7-c_1)^2(\pi^7-c_2))^3t^{-3}+\cdots.$$
All that 7 roots of $(\pi^7-c_1)$ are major, each of them extends  to a final major $\pi$-root $\sigma_{1i}$ with split order $\delta_1$:
$$f(\sigma_{1i})=a_{0i}\pi (\pi^3-a_{1i}) t^{-2/21}+\cdots, \quad g(\sigma_{1i})=b_{0i} (\pi^3-b_{1i})(\pi^3-b_{2i} ) t^{-3/21}+\cdots $$
for $  i=1,\dots,7.$

\noindent All that 7 roots of $(\pi^7-c_2)$ are minor, each of them extends to a final minor $\pi$-root with order $9/7$. We also have  one principal final minor $\pi$-root with split order $V_3/U_3=3$. Therefore,
from final minor roots, we have $ {\I}_m(f , g) = 1+7\times (9/7-1)+(3-1)=5$, while final major roots give
$ {\I}_M(f, g)=4 \times 7 \times 3/21=4$, and that is a contradiction. So we can rule out this case.

\item \label{L:case84 ii}
$$n=84, \quad m=56, \quad M_2=72, \quad M_3=82, \quad V_3=3, \quad V_2=5, \quad \delta_2=1/4, \quad \delta_1=7/12.$$
$$f(\sigma_2)=a((\pi^4-c_1)^5\pi)^2 t^{-14/4}+\cdots, \quad g(\sigma_2)=b((\pi^4-c_1)^5\pi)^3 t^{-21/4}+\cdots.$$
In this case, 4 roots from $(\pi^4-c_1)$ are major and extend to final major $\pi$-roots $\sigma_{1i}$ with split order $\delta_1$:
$$f(\sigma_{1i})=a_{0i}\pi (\pi^3-a_{1i}) (\pi^3-a_{2i}) (\pi^3-a_{3i}) t^{-2/12}+\cdots, $$
$$ g(\sigma_{1i})=b_{0i} (\pi^3-b_{1i})(\pi^3-b_{2i} )(\pi^3-b_{3i})(\pi^3-b_{4i} ) (\pi^3-b_{5i}) t^{-3/12}+\cdots $$
for $  i=1,\dots,4;$
while 1 root from $(\pi-0)$ is minor and extend to final minor $\pi$-root with split order  $2$. So we have
$ {\I}_m(f , g)= 1+(2-1)+(3-1)=4$ from final minor roots while
$ {\I}_M(f , g)=4 \times 10 \times 3/12=10$. We can not rule out this case.
\end{romanenum}

\section{Splitting  of the Principal Minor Roots}
In this section we will show that the principal minor roots are not splitting at order 1 for all effective quasi-approximate roots.
We need a few preliminaries first.
\subsection{Planar semigroups}  For Abhyankar-Moh planar semigroup theory and characteristic $\delta$-sequence, we refer \cite{am}, \cite{ss} and \cite{xu2}. We recall definitions of $q$-sequence and $M$-sequence
$$M_1=-\delta_1$$
$$q_i=\delta_{i-1} \frac{d_{i-1}}{d_i}-\delta_i, \quad M_i=M_{i-1}+q_i $$
for $ i=2,\dots,h$.

For an integer subset $A$, $\Gamma(A)$ denotes the semigroup generated by all elements in $A$.
\begin{lemma}\label{L:semigroup}
Let  $\delta=(\delta_0, \dots,\delta_h)\; (h\ge 2)$ be a characteristic $\delta$-sequence and $k$ be an integer with $2\le k \le h$. Then we have
\begin{romanenum}
\item \label{L:semigroup i}
$\delta_k+M_k \in \Gamma (\delta_1,\dots,\delta_{k-1})$,
\item \label{L:semigroup ii}
$\delta_k+M_k -\delta_0 \notin \Gamma (\delta_0,\dots,\delta_{k-1})$.
\end{romanenum}
\end{lemma}

\begin{proof}
$\delta_1 +M_1=0$.
As 
$$\delta_k+M_k={\delta_{k-1}}\frac{d_{k-1}}{d_k}-q_k+M_k=\delta_{k-1}(\frac{d_{k-1}}{d_k}-1)+\delta_{k-1}+M_{k-1}$$
by induction, we have
 $$\delta_k+M_k=\delta_{k-1}(\frac{d_{k-1}}{d_k}-1)+\cdots +\delta_{1}(\frac{d_{1}}{d_2}-1)$$
(\ref{L:semigroup i}) is proved.
Now
$$\delta_k+M_k-\delta_0=\delta_{k-1}(\frac{d_{k-1}}{d_k}-1)+\cdots +\delta_{1}(\frac{d_{1}}{d_2}-1)-\delta_0$$
is a standard expansion by the $\delta$-sequence, by \cite{xu2} Lemma 2.3, (\ref{L:semigroup ii}) is proved.
\end{proof}

\subsection{One kind of special ordinary differential equations of polynomials}  After analyzing the proof of the second part of \cite{moh} Proposition A.2, we have

\begin{lemma}\label{L:semigroupn}
Let  $\deg p(\pi)=m$ and $\deg q(\pi)=(l-1)m+1$. Suppose $p(\pi)$, $q(\pi)$ satisfy the following equation
$$D(m, m(l-1), p(\pi), q(\pi))=cp(\pi)^l$$
where $c\in \KK^*$. Then $q(\pi)=\frac{c}{m} (\pi-a) p(\pi)^{l-1}$ for some $a\in \KK$.
\end{lemma}

\subsection{Quasi-approximate roots of a polynomial curve }  Let $(f, g)$ be a Jacobian pair,  monic in $y$ and $f$ has two points at infinity. Let $\deg f=m$ and $\deg g=n$. Consider $y$ as a parameter, $(f, g)$ defines a polynomial curve. From \cite{moh}, there are effective index  $s>2$ and quasi-approximate roots
$T_i(f, g)\in \KK[f, g]$  
with $T_0=g$, $T_1=f$ and $\deg T_i=-\mu_i$ for $i=0, \dots,s$.

Let $d_i=\gcd (-\mu_0, \dots,-\mu_{i-1}) $ for $ i=1,\dots, s+1$. Then $( -\mu_0/d_{s+1},\dots,-\mu_s/d_{s+1})$  is a $\delta$-sequence.
Let the  multiplicity of the principal minor roots of $f$ is $m\frac{u_s}{d_s}$, here $u_s<v_s=d_s-u_s$. Then the multiplicities of the principal minor roots of $ T_0, \dots, T_{s-1}, T_s$ are
$$ -\mu_0\frac{u_s}{d_s},\dots,-\mu_{s-1}\frac{u_s}{d_s}, (-\mu_s-2)\frac{u_s}{d_s}+1$$ respectively.

\begin{prop}\label{P:principalminorrt}
 Let  $\sigma_1$ be the $\pi$-root of order $1$ for the principal minor roots of $f$. Let

\begin{align*}
T_0(\sigma_1)&=T_{0,\sigma_1}(\pi) t^{\frac{v_s-u_s}{d_s}\mu_0}+\cdots\\
&\vdots\\
T_{s-1}(\sigma_1)&=T_{s-1,\sigma_1}(\pi) t^{\frac{v_s-u_s}{d_s}\mu_{s-1}}+\cdots\\
T_{s}(\sigma_1)&=T_{s,\sigma_1}(\pi) t^{\frac{v_s-u_s}{d_s}(\mu_s+2)}+\cdots
\end{align*}
Then $ T_{0,\sigma_1}(\pi),\dots,T_{s,\sigma_1}(\pi)$ are powers of a common linear polynomial.
\end{prop}

\begin{proof}
Using the proof of \cite{moh} Proposition 4.6 verbatim, we get
$$ T_{0,\sigma_1}(\pi)=C_0p(\pi)^{\frac{-\mu_0}{d_s}},\dots,T_{s-1,\sigma_1}(\pi)=C_{s-1}p(\pi)^{\frac{-\mu_{s-1}}{d_s}}$$
where $\deg p(\pi)=u_s$.

Finally by Lemma \ref{L:semigroupn}, we have 
$$T_{s,\sigma_1}(\pi)=C_s p(\pi)^{\frac{-\mu_s-2}{d_s}}(\pi-a)$$
where $a\in \KK$, $C_i\in \KK^*$ for $i=0,\dots,s$.

We claim $p(\pi)=(\pi-a)^{u_s}$. If not, let $\pi=b$ be a root of $p(\pi)$ and $b\ne a$. We extend $\sigma_1|_{\pi=b}$ to 
a final minor $\pi$-root $\sigma_0$ with split order $\delta_0$. 

First we point out that for any $1<\delta<\delta_0$ and $\pi$-root $\sigma=\sigma_0|_{<\delta}+\pi t^\delta$ we have that for some $\lambda<0$,
\begin{align*}
\ord T_0(\sigma)&=(-\mu_0) \lambda\\
&\vdots\\
\ord T_{s-1}(\sigma)&=(-\mu_{s-1}) \lambda\\
\ord T_s(\sigma)&=(-\mu_s-2) \lambda
\end{align*}
Using the proof of \cite{moh} Proposition 4.2 verbatim, we conclude that $\sigma$ is a distribution detector for $T_0(y),\dots,T_s(y)$.

Hence we have

\begin{align*}
T_0(\sigma_0)&=T_{0,\sigma_0}(\pi) t^0+\cdots\\
&\vdots\\
T_{s-1}(\sigma_0)&=T_{s-1,\sigma_0}(\pi) t^0+\cdots\\
T_{s}(\sigma_0)&=T_{s,\sigma_0}(\pi) t^0+\cdots
\end{align*}
where
\begin{align*}
\deg T_{0,\sigma_0}(\pi)&=\frac{-\mu_0}{d_s} r\\
&\vdots\\
\deg T_{s-1,\sigma_0}(\pi)&=\frac{-\mu_{s-1}}{d_s}r\\
\deg T_{s,\sigma_0}&=\frac{-\mu_{s}-2}{d_s} r
\end{align*}
for some positive integer $r$ and $r$ satisfies $1\le r \le u_s$.

We consider polynomial curve $(f_{\sigma_0} (\pi), g_{\sigma_0}(\pi))$ with $\pi$ as the parameter. Then $\{T_{i,\sigma_0}(\pi)\}_{0 \le i \le s-1}$  are quasi-approximate roots of the curve. 
$$T_{s,\sigma_0}(\pi)=T_s(f_{\sigma_0}(\pi), g_{\sigma_0}(\pi)) \in \KK [f_{\sigma_0}(\pi), g_{\sigma_0}(\pi)]$$
As
$\gcd (\frac{-\mu_0}{d_s}r,\dots,\frac{-\mu_{s-1}}{d_s}r)=r$ and $r$ divides 
$\deg T_{s,\sigma_0}(\pi)=\frac{-\mu_s-2}{d_s} r$, we conclude that 
$\frac{-\mu_{s}-2}{d_s} r$  is in the semigroup $\Gamma (\frac{-\mu_0}{d_s}r,\dots,\frac{-\mu_{s-1}}{d_s}r)$.

On the other hand, $$-\mu_s-2=-\mu_s+M_s-n \notin \Gamma(-\mu_0, \dots, -\mu_{s-1})$$
by Lemma \ref{L:semigroup}. This is a contradiction.
\end{proof}

\begin{remark} By \cite{ML} the coefficient of $t$ in any root $\alpha$ of $f(y)$ is zero. We do not find this information can simplify the proof that $T_s$ is not split at order $1$.
\end{remark}

\begin{cor} \label{cor:c7}Suppose $u_s>1$ and let
$\sigma$ be a $\pi-$root of $f$ for the principle minor roots. If order $\delta_\sigma < \dfrac{v_s+1}{u_s+1}$ then $\sigma$ can not split to $u_s$ different roots.
\end{cor}

\begin{proof}
Suppose $\sigma$ split to $u_s$ different roots. For simplify we denote $\delta_\sigma=\delta$. First from last Proposition we know $\delta>1$. We can write
$$f(\sigma)=a_0p(\pi)^{m/d_s}t^{-v_s\frac{m}{d_s}+ u_s\frac{m}{d_s}\delta}+\cdots$$
here 
$$p(\pi)=(\pi-c_1)\cdots(\pi-c_{u_s})$$
{\rm with} 
$$c_i \ne c_j \,\; {\rm if }\;\;i\ne j$$
From above Proposition, we know that $\delta>1$ and we can write
$$g(\sigma)=a_1p(\pi)^{n/d_s}t^{-v_s\frac{n}{d_s}+ u_s\frac{n}{d_s}\delta}+\cdots$$
$$(T_s)_f(\sigma)=bp(\pi)^{(-\mu_s+n-2)/d_s}t^{-v_s\frac{-\mu_s+n-2}{d_s}+ u_s\frac{-\mu_s+n-2}{d_s}\delta}+\cdots$$
$$T_s(\sigma)=a_sq(\pi)t^{-v_s\frac{-\mu_s+n-2}{d_s}-1+(u_s\frac{-\mu_s+n-2}{d_s}+1)\delta}+\cdots$$
here $\deg q(\pi)=u_s\frac{-\mu_s+n-2}{d_s}+1$.
By \cite{moh} [Proposition 4.1], we have
\begin{align}\label{eq:c7}
\dfrac{\partial (T_s(\sigma),g(\sigma))}{\partial(t, \pi)}=-J{\cdot}(T_s)_f(\sigma)t^{-2+\delta}
\end{align}
With the above information, and note that 
$${-v_s\frac{-\mu_s+n-2}{d_s}-1+(u_s\frac{-\mu_s+n-2}{d_s}+1)\delta}<-v_s-1+(u_s+1)\delta<0,$$
so every root of $p(\pi)$ is a root of $q(\pi)$ and $p(\pi)$ has no multi roots, we have
$q(\pi)=p(\pi)q_1(\pi)$. We can use same reason to $q_1(\pi)$ and on. Finally we have 
$$q(\pi)=p(\pi)^ {\frac{-\mu_s+n-2}{d_s}}(\pi-c)$$
for some $c$.
the above equation \ref{eq:c7} can be simplified to
$$\dfrac{\partial}{\partial (t, \pi)}(p(\pi)t^{-v_s+u_s}, (\pi-c)t^{-1+\delta})=Cp(\pi)t^{-v_s+u_s-2+\delta}.$$
But this equation can not hold and it is a contradiction. 
\end{proof}

\section{Case (99, 66) and split of principal minor roots}

Moh writes in \cite{moh} page 209 on the principal minor roots, ``Namely, there is an [sic] unique $\pi$-root $\sigma$ of $g(y)$ of the following form
$$
\sigma=a_{-1}t^{-1}+a_0+a_1t+\pi t^2 
$$
with
$$g(\sigma)=g_\pi(\sigma)t^{-18}+\cdots \quad [{\rm sic,}\; g_\pi(\sigma) \;{\rm should\; be\;} g_\sigma(\pi)]$$
The polynomial $g_\sigma(\pi)$ is either a power of a linear polynomial or the 9-th power of a cubic polynomial with precisely two roots."

For a long time, we try to give judgement for this statement. To do so, there are three claims have to be proven.
\begin{romanenum}
\item\label{case:c1} The principal minor roots do not split except at $\pi$-root order $\delta=2$.
\item\label{case:c2}   $g_\sigma(\pi)$ can not have three roots when $\delta=2$.
\item\label{case:c3}   Let $\delta=2$ and let $g_\sigma(\pi)=c(\pi-b_1)^{18} (\pi-b_2)^9$. Roots in the disc with the center $\sigma |_{\pi=b_1}$ do not split before final.
\end{romanenum}
We can not find materials to support these claims in his paper. We sent Moh an email to point out the possible gaps.  He replied with, ``I will make an investigation of the issue and reply to your e-mail as soon as possible." on Jan. 06, 2016. 

We show that (\ref{case:c2}) and (\ref{case:c3}) are true, but there is one exception case for (\ref{case:c1}) open.

Our tools are two constrains: one is that the denominator of order $\delta \le u_s=3$; another is equation \ref{eq:c7}.

From \cite{moh} page 202, we have data
\begin{align}\label{data:c99}
n=99, \quad m=66, \quad M_2=77, \quad M_3=97, \quad V_3=8, \quad V_2=8, \quad \delta_2=1/3, \quad \delta_1=4/9.
\end{align}
Let $\sigma$ be $\pi$-root of the principal minor roots. From data \ref{data:c99} we get that $-\mu_2=55$ and 
$-\mu_3=145$. We have
$$f(\sigma)=p(\pi)^6t^{6(-8+3\delta)}+\cdots$$
$$g(\sigma)=p(\pi)^9t^{9(-8+3\delta)}+\cdots$$
$$T_2(\sigma)=p(\pi)^5t^{5(-8+3\delta)}+\cdots$$
$$(T_3)_f(\sigma)=p(\pi)^{22}t^{22(-8+3\delta)}+\cdots$$
$$T_3(\sigma)=q(\pi)t^{13(-8+3\delta)-1+\delta}+\cdots$$
here $\deg p(\pi)=u_3=3$ and $\deg q(\pi)=13*3+1$.

(\ref{case:c1}) By the constrain of the denominator of $\delta$, if $\delta$ splits, it must split to 2 or 3 roots. 

When $\delta \ne 2$ and $\delta \ne \frac{5}{2}$. From equation \ref{eq:c7}, we can step by step deduce 
$$q(\pi)=p(\pi)^{13}(\pi-c)$$
and the equation \ref{eq:c7} can not hold for it.

When $\delta=2$. As there is equation
\begin{align}\label{eq:c8}
\frac{\partial}{\partial (t, \pi)}(\pi(\pi+3a)^2(\pi-2a)t^{-1}, \pi^2(\pi+3a)t^{-2})=5\pi^4(\pi+3a)^2 t^{-4},
\end{align}
The equation \ref{eq:c7} produces Moh's 2 roots split case.

When $\delta=\frac{5}{2}$. From equation \ref{eq:c7}, we can deduce
$$q(\pi)=p(\pi)^{10}q_1(\pi)$$
here $\deg q_1(\pi)=3*3+1$.
The equation \ref{eq:c7} gives us
$$\frac{\partial}{\partial (t, \pi)}(q_1(\pi), p(\pi) t^{-\frac{1}{2}})=-p(\pi)^4 t^{-\frac{3}{2}}.$$
But this can hold for any $p(\pi)$ and $q_1(\pi)=-2\int p(\pi)^3dx$. So there is possible $p(\pi)=\pi(\pi^2-c)$. Thus this case is open and it suggests case (99, 66) is open.

(\ref{case:c2}) This case is proved by Corollary \ref{cor:c7}.

(\ref{case:c3}) After $\pi$-root extends to $\delta>2$ and $\delta<3$ along $\sigma |_{\pi=b_1}$, we have
$$f(\sigma)=p(\pi)^6t^{12(-3+\delta)}+\cdots$$
$$g(\sigma)=p(\pi)^9t^{18(-3+\delta)}+\cdots$$
$$T_2(\sigma)=p(\pi)^5t^{10(-3+\delta)}+\cdots$$
$$T_3(\sigma)=q(\pi)t^{25(-3+\delta)}+\cdots$$
here $\deg p(\pi)=2$, $\deg q(\pi)=25$ and $\dfrac{\partial }{\partial(t, \pi)}(q(\pi)t^{25(-3+\delta)},p(\pi)^9t^{18(-3+\delta)})=0$.

As $\gcd (12, 18, 10, 25)=1$, $p(\pi)$ and $q(\pi)$ are powers of a common linear factor $(\pi-c)$, i.e. $\pi$-root is not split at this order $\delta$. The claim is proved.

\section*{Acknowlegement}
The author thanks Professor Mattias Josson and Professor Christian Valqui for pointing out errors in previous version.
%-----------------------------------------------------------------------------------------------------------------------------------------------------------

%\bibliographystyle{plain}
%\bibliography{//Users/zywina/Documents/papers/bib/master}


\end{document}